\documentclass{proc-l-hijacked}
\usepackage{amssymb, amsmath, amsthm, hyperref,}

\newtheorem{theorem}{Theorem}[section]
\newtheorem{lemma}[theorem]{Lemma}

\newtheorem{corollary}[theorem]{Corollary}

\theoremstyle{definition}

\numberwithin{equation}{section}

\begin{document}

\title[Uniqueness under Spectral Variation in the Socle]{ Uniqueness under Spectral Variation in the Socle of a Banach Algebra }
\author{F. Schulz \and R. Brits}
\address{Department of Pure and Applied Mathematics, University of Johannesburg, South Africa}
\email{francoiss@uj.ac.za, rbrits@uj.ac.za}

\begin{abstract}
Let $A$ be a complex semisimple Banach algebra with identity, and denote by $\sigma'(x)$ and $\rho (x)$ the nonzero spectrum and spectral radius of an element $x \in A$, respectively. We explore the relationship between elements $a, b \in A$ that satisfy one of the following conditions: (1) $\sigma' (ax) \subseteq \sigma' (bx)$ for all $x \in A$, (2) $\rho (ax) \leq \rho (bx)$ for all $x \in A$. The latter problem was identified by Bre\v{s}ar and \v{S}penko in \cite{detelements}. In particular, we use these conditions to spectrally characterize prime Banach algebras amongst the class of Banach algebras with nonzero socles, as well as to obtain spectral characterizations of socles which are minimal two-sided ideals.  
\end{abstract}
	\parindent 0mm
	
	\maketitle

\section{Introduction}

By $A$ we denote a complex Banach algebra with identity element $\mathbf 1$ and invertible group $G(A)$. Moreover, it will be assumed throughout that $A$ is semisimple (i.e. the Jacobson radical of $A$, denoted $\mathrm{Rad}\:A$, only contains $0$). We will write $Z(A)$ for the center of $A$, that is, for the set of all $x \in A$ such that $xy=yx$ for all $y \in A$. For $x\in A$ we denote by $\sigma_A(x) =\left\{\lambda\in\mathbb C:\lambda\mathbf{1}-x\notin G(A)\right\}$, $\rho_{A} (x) = \sup\left\{\left|\lambda\right| : \lambda \in \sigma_{A} (x)\right\}$ and $\sigma_{A} '(x)=\sigma_A(x) -\{0\}$ the spectrum, spectral radius and nonzero spectrum of $x$, respectively. If the underlying algebra is clear from the context, then we shall agree to omit the subscript $A$ in the notation $\sigma_A(x)$, $\rho_{A} (x)$ and $\sigma_{A} '(x)$. This convention will also be followed in some of the forthcoming definitions. We shall also agree to reserve the notation $\cong$ exclusively for algebra isomorphisms. Moreover, we recall that an element $x$ of $A$ is called quasinilpotent if $\sigma (x) = \left\{0 \right\}$.\\

In \cite{detelements} M. Bre\v{s}ar and \v{S}. \v{S}penko consider two interesting problems which resulted from certain questions centered around Kaplansky's problem on spectrum preserving maps \cite{kaplansky1970algebraic}:\\

\noindent\emph{Problem} 1. Suppose that $a, b \in A$ satisfy $\sigma (ax) = \sigma (bx)$ for all $x \in A$. Does this imply $a = b$?\\

\noindent\emph{Problem} 2. Suppose that $a, b \in A$ satisfy
\begin{equation} 
\rho (ax) \leq \rho (bx) \;\,\mathrm{for\;all}\;\, x \in A.
\label{eq1.1}
\end{equation}
What is the relation between $a$ and $b$?\\

The first problem has been settled by G. Braatvedt and R. Brits in \cite{uniqueness}:

\begin{theorem}\cite[Theorem 2.1, Theorem 2.6]{uniqueness}
	\label{1.1}
	Let $a, b \in A$. Then the following are equivalent:
	\begin{itemize}
		\item[\textnormal{(i)}]
		$a = b$.
		\item[\textnormal{(ii)}]
		$\sigma (ax) = \sigma (bx)$ for all $x \in A$ such that $\rho \left(x-\mathbf{1}\right)<1$.
		\item[\textnormal{(iii)}]
		$\sigma (a+x) = \sigma (b+x)$ for all $x$ in some open neighbourhood of $-b$.
	\end{itemize}
\end{theorem}

Problem 2, as to be expected, is slightly more intricate. Evidence such as \cite[Example 3.3]{detelements} suggests that the answer to this question may depend on the algebra or on the elements under consideration. Indeed, in the special situation where $b = \mathbf{1}$ it was found in \cite{scalars} that $a$ must then belong to $Z(A)$. Moreover, in \cite{detelements} Bre\v{s}ar and \v{S}penko investigated the special case where $A$ is a prime $C^{\ast}$-algebra. The conclusion in this case is that the elements $a$ and $b$ satisfying (\ref{eq1.1}) are necessarily linearly dependent. We recall that $A$ is a \textit{prime algebra} if all nonzero two-sided ideals $I$ and $J$ of $A$ satisfy $IJ \neq \left\{0 \right\}$. In particular, we will see that the linear dependence obtained in the prime $C^{\ast}$-algebra case extends to the case where $A$ is assumed to be prime with a nonzero socle. Furthermore, the consideration of Problem 2 leads to spectral characterizations of socles which are minimal two-sided ideals. Other characterizations of such socles were recently obtained by the authors and G. Braatvedt (cf. \cite[Theorem 3.8, Theorem 3.9]{tracesocleident} and \cite[Theorem 4.4]{commutatorsinsocle}).\\   

The notions of rank, trace and determinant are well-established for operator theory. Moreover, in a more general setting, these notions provide an analytic means to investigate the socle of a semisimple Banach algebra. This latter idea was made precise by B. Aupetit and H. Du. T. Mouton in \cite{aupetitmoutontrace} where they managed to show that these notions can be developed, without the use of operators, in a purely spectral and analytic manner. This paper is fundamental to our discussion here, so as in \cite{tracesocleident} we briefly summarize some of the theory in \cite{aupetitmoutontrace} before we proceed.\\

For each nonnegative integer $m$, let
$$\mathcal{F}_{m} = \left\{a \in A: \#\sigma'(xa) \leq m \;\,\mathrm{for\;all}\;\,x \in A \right\},$$
where the symbol $\#K$ denotes the number of distinct elements in a set $K\subseteq \mathbb C$. Following Aupetit and Mouton in \cite{aupetitmoutontrace}, we define the \textit{rank} of an element $a$ of $A$ as the smallest integer $m$ such that $a \in \mathcal{F}_{m}$, if it exists; otherwise the rank is infinite. In other words,
$$\mathrm{rank}\,(a) = \sup_{x \in A} \#\sigma'(xa).$$
If $a \in A$ is a finite-rank element, then
$$E(a) = \left\{x \in A : \#\sigma'(xa) = \mathrm{rank}\,(a) \right\}$$
is a dense open subset of $A$ \cite[Theorem 2.2]{aupetitmoutontrace}. A finite-rank element $a$ of $A$ is said to be a \textit{maximal finite-rank element} if $\mathrm{rank}\,(a) = \#\sigma'(a)$. With respect to $\mathrm{rank}$ it is useful to know results such as Jacobson's Lemma \cite[Lemma 3.1.2]{aupetit1991primer}, the Spectral Mapping Theorem \cite[Theorem 3.3.3]{aupetit1991primer}) and the Scarcity Theorem \cite[Theorem 3.4.25]{aupetit1991primer}. It can be shown \cite[Corollary 2.9]{aupetitmoutontrace} that the \textit{socle}, written $\mathrm{Soc}\:A$, of a semisimple Banach algebra $A$ coincides with the collection $\bigcup_{m = 0}^{\infty} \mathcal{F}_{m}$ of finite rank elements. We mention a few elementary properties of the rank of an element \cite[p. 117]{aupetitmoutontrace}. Firstly, $\#\sigma'(a) \leq \mathrm{rank}\,(a)$ for all $a \in A$. Furthermore, $\mathrm{rank}\,(xa) \leq \mathrm{rank}\,(a)$ and $\mathrm{rank}\,(ax)\leq \mathrm{rank}\,(a)$ for all $x, a \in A$, with equality if $x \in G(A)$. Moreover, the rank is lower semicontinuous on $\mathrm{Soc}\:A$. It is also subadditive, i.e. $\mathrm{rank}\,(a+b) \leq \mathrm{rank}\,(a) + \mathrm{rank}\,(b)$ for all $a, b \in A$ \cite[Theorem 2.14]{aupetitmoutontrace}. Finally, if $p$ is a projection of $A$, then $p$ has rank one if and only if $p$ is a minimal projection, that is, if $pAp = \mathbb{C}p$ \cite[p. 117]{aupetitmoutontrace}. It is also worth mentioning here that a projection $p$ is minimal if and only if $Ap$ is a nontrivial left ideal which does not contain any left ideals other than $\left\{0 \right\}$ and itself, that is, if and only if $Ap$ is a nontrivial minimal left ideal \cite[Lemma 30.2]{bonsall1973complete}. A similar result holds true for the right ideal $pA$. We will also define a \textit{minimal two-sided ideal} in this manner, that is, as a two-sided ideal which does not contain any two-sided ideals other than $\left\{0 \right\}$ and itself.\\

The following result is fundamental to the theory developed in \cite{aupetitmoutontrace} and is mentioned here for convenient referencing later on:\\

\noindent\textbf{Diagonalization Theorem} \cite[Theorem 2.8]{aupetitmoutontrace}: Let $a \in A$ be a nonzero maximal finite-rank element and denote by $\lambda_{1}, \ldots, \lambda_{n}$ its nonzero distinct spectral values. Then there exists $n$ orthogonal minimal projections $p_{1}, \ldots, p_{n} \in Aa \cap aA$ such that 
$$a = \lambda_{1}p_{1} + \cdots + \lambda_{n}p_{n}.$$

In particular, the Diagonalization Theorem easily implies the well-known result that every element of the socle is \emph{Von Neumann regular}, that is, for each $a \in \mathrm{Soc}\:A$, there exists an $x \in \mathrm{Soc}\:A \subseteq A$ such that $a = axa$ \cite[Corollary 2.10]{aupetitmoutontrace}.\\

If $a \in \mathrm{Soc}\:A$ we define the \textit{trace} of $a$ as in \cite{aupetitmoutontrace} by
$$\mathrm{Tr}\,(a) = \sum_{\lambda \in \sigma (a)} \lambda m\left(\lambda, a\right),$$
where $m(\lambda,a)$ is the \emph{multiplicity of $a$ at $\lambda$}. A brief description of the notion of multiplicity in the abstract case goes as follows (for particular details one should consult \cite{aupetitmoutontrace}): Let $a \in \mathrm{Soc}\:A$, $\lambda\in\sigma(a)$ and let $B(\lambda,r)$ be an open disk centered at $\lambda$ such that $B(\lambda,r)$ contains no other points of $\sigma(a)$. It can be shown \cite[Theorem 2.4]{aupetitmoutontrace} that there exists an open ball, say $U\subseteq A$, centered at $\mathbf{1}$ such that $\#\left[\sigma(xa)\cap B(\lambda,r) \right]$ is constant as $x$ runs through $E(a)\cap U$. This constant integer is the multiplicity of $a$ at $\lambda$. It can also be shown that $m\left(\lambda,a\right) \geq 1$ and
\begin{equation}
\sum_{\alpha \in \sigma (a)} m(\alpha, a) = \left\{\begin{array}{cl} 1+\mathrm{rank}\,(a) & \mathrm{if}\;\,0 \in \sigma (a) \\
\mathrm{rank}\,(a) & \mathrm{if}\;\,0 \notin \sigma (a).
\end{array}\right.
\label{beq0}
\end{equation} 
Furthermore, we note that the trace has the following useful properties:
\begin{itemize}
	\item[(i)]
	$\mathrm{Tr}$ is a linear functional on $\mathrm{Soc}\:A$ (\cite[Theorem 3.3]{aupetitmoutontrace} and \cite[Lemma 2.1]{tracesocleident}).
	\item[(ii)]
	$\mathrm{Tr}\,(ab) = \mathrm{Tr}\,(ba)$ for each $a \in \mathrm{Soc}\:A$ and $b \in A$ \cite[Corollary 2.5]{tracesocleident}.
	\item[(iii)]
	For any $a \in A$, if $\mathrm{Tr}\,(ax) = 0$ for each $x \in \mathrm{Soc}\:A$, then $a\mathrm{Soc}\:A = \left\{0 \right\}$. Moreover, if $a \in \mathrm{Soc}\:A$, then $a=0$ \cite[Corollary 3.6]{aupetitmoutontrace}.
	\item[(iv)]
	If $f$ is an analytic function from a domain $D$ of $\mathbb{C}$ into $\mathrm{Soc}\:A$, then $\lambda \mapsto \mathrm{Tr}\,\left(f(\lambda)\right)$ is holomorphic on $D$ \cite[Theorem 3.1]{aupetitmoutontrace}.
\end{itemize}

Let $\lambda \in \sigma (a)$ and suppose that $B(\lambda,2r)$ separates $\lambda$ from the remaining spectrum of $a$. Let $f_{\lambda}$ be the holomorphic function which takes the value $1$ on $B(\lambda,r)$ and the value $0$ on $\mathbb{C} - \overline{B}(\lambda,r)$. If we now let $\Gamma_{0}$ be a smooth contour which surrounds $\sigma (a)$ and is contained in the domain of $f_{\lambda}$, then
$$p\left(\lambda, a\right) = f_{\lambda} (a) = \frac{1}{2\pi i}\int_{\Gamma_{0}} f_{\lambda} (\alpha) \left(\alpha \mathbf{1}-a\right)^{-1}\,d\alpha$$
is referred to as the \textit{Riesz projection} associated with $a$ and $\lambda$. By the Holomorphic Functional Calculus, Riesz projections associated with $a$ and distinct spectral values are orthogonal, all commute with $a$ and for $\lambda \neq 0$
\begin{equation}
p\left(\lambda, a\right) = \frac{a}{2\pi i} \int_{\Gamma_{0}} \frac{f_{\lambda}\left(\alpha\right)}{\alpha} \left(\alpha\mathbf{1} - a\right)^{-1}\,d\alpha \in Aa \cap aA. \label{beq1}
\end{equation}
It is also worth mentioning that the orthogonal minimal projections obtained in the conclusion of the Diagonalization Theorem are in fact the Riesz projections of the maximal finite-rank element associated with each of its corresponding nonzero spectral values.\\

In the operator case, $A=B(X)$ (bounded linear operators on a Banach space $X$), the ``spectral" rank and trace both coincide with the respective classical operator definitions.

\section{Uniqueness under Spectral Variation in the Socle}

Let $a \in A$. J. Zem\'{a}nek has shown that $\rho(a+x) = 0$ for all quasinilpotent $x$ in $A$ if and only if $a \in \mathrm{Rad}\:A$ \cite[Theorem 5.3.1]{aupetit1991primer}. In order to get some feeling for the subject matter, we start by utilizing the aforementioned result to show that condition (iii) in Theorem \ref{1.1} can be substantially relaxed:

\begin{theorem}\label{3.1}
	Let $a, b \in A$. Then the following are equivalent:
	\begin{itemize}
		\item[\textnormal{(i)}]
		$a = b$.
		\item[\textnormal{(ii)}]
		$\rho \left(a + x\right) \leq \rho \left(b+x\right)$ for all $x$ in some open neighbourhood of $-b$.
	\end{itemize}
\end{theorem}

\begin{proof}
	Certainly, (i) $\Rightarrow$ (ii). We therefore proceed to show that (ii) $\Rightarrow$ (i). We claim that $\rho \left(a - b + q\right) = 0$ for all quasinilpotent elements $q$ in $A$: Let $q$ be any quasinilpotent element in $A$. Consider the analytic function $f: \mathbb{C} \rightarrow A$ defined by $f(\lambda) = a-b+ \lambda q$. By hypothesis and the Spectral Mapping Theorem, there exists a real number $k > 0$ such that $\rho \left(a - b + \lambda q\right) \leq \rho \left(\lambda q\right) = 0$ whenever $\left|\lambda\right| < k$. Hence, $\sigma \left(f\left(\lambda\right)\right) = \left\{0 \right\}$ whenever $\left|\lambda\right| < k$. By the Scarcity Theorem we may therefore conclude that $\sigma \left(f\left(\lambda\right)\right) = \left\{\alpha \left(\lambda\right) \right\}$ for all $\lambda \in \mathbb{C}$, where $\alpha$ is a mapping from $\mathbb{C}$ into $\mathbb{C}$. By \cite[Corollary 3.4.18]{aupetit1991primer}, $\alpha$ is an entire function. However, $\alpha \left(\lambda\right) = 0$ whenever $\left|\lambda\right| < k$, and so, from basic Complex Analysis it must be the case that $\alpha \left(\lambda\right) = 0$ for all $\lambda \in \mathbb{C}$. This proves our claim. Consequently, $a- b \in \mathrm{Rad}\:A$ by \cite[Theorem 5.3.1]{aupetit1991primer}. Thus, by semisimplicity we have the result.
\end{proof}

\begin{theorem}\label{3.2}
	Let $a, b \in \mathrm{Soc}\:A$. Then $a= b$ if and only if any one of the following holds true:
	\begin{itemize}
		\item[\textnormal{(i)}]
		$\sigma (ax) = \sigma (bx)$ for all rank one elements $x \in A$.
		\item[\textnormal{(ii)}]
		$\sigma (a+x) = \sigma (b+x)$ for all rank one elements $x \in A$.
	\end{itemize}
\end{theorem}

\begin{proof}
	Obviously, if $a=b$ then conditions (i) and (ii) both hold. So let $a, b \in \mathrm{Soc}\:A$ and assume that condition (i) holds. Then $\mathrm{Tr}\,(ax) = \mathrm{Tr}\,(bx)$ for all rank one elements $x \in A$. Let $y \in \mathrm{Soc}\:A$ be arbitrary. Clearly, $\mathrm{Tr}\,(ay) = \mathrm{Tr}\,(by)$ if $y=0$. So assume that $y \neq 0$. By the Diagonalization Theorem and the density of $E(y)$ there exist rank one projections $p_{1}, \ldots, p_{n}$, $\alpha_{1}, \ldots, \alpha_{n} \in \mathbb{C}$ and a $u \in G(A)$ such that $y = \alpha_{1}up_{1} + \cdots + \alpha_{n}up_{n}$. Thus, by the linearity of the trace we readily obtain $\mathrm{Tr}\,(ay) = \mathrm{Tr}\,(by)$ for all $y \in \mathrm{Soc}\:A$. Consequently, $\mathrm{Tr}\,\left(\left(a-b\right)y\right) = 0$ for all $y \in \mathrm{Soc}\:A$. Thus, since $a-b \in \mathrm{Soc}\:A$, it follows from \cite[Corollary 3.6]{aupetitmoutontrace} that $a-b = 0$. Next take $a, b \in \mathrm{Soc}\:A$ and assume that condition (ii) holds. Fix any $\lambda \notin \sigma (a) \cup \sigma (b)$ and $0 \neq \alpha \in \mathbb{C}$. If $x \in A$ has rank one, then we have
	$$\lambda \mathbf{1} - \left(a + \alpha^{-1}x\right) \in G(A) \Leftrightarrow \lambda \mathbf{1} - \left(b + \alpha^{-1}x\right) \in G(A).$$
	Consequently,
	$$\left(\lambda \mathbf{1}-a\right)\left(\mathbf{1} + \left(\lambda \mathbf{1}-a\right)^{-1}\alpha^{-1}x\right) \in G(A) \Leftrightarrow \left(\lambda \mathbf{1}-b\right)\left(\mathbf{1} + \left(\lambda \mathbf{1}-b\right)^{-1}\alpha^{-1}x\right) \in G(A).$$
	Since the first term on the left of each expression is invertible, it follows that
	$$\alpha \in \sigma \left(\left(\lambda \mathbf{1}-a\right)^{-1}x\right) \Leftrightarrow \alpha \in \sigma \left(\left(\lambda \mathbf{1}-b\right)^{-1}x\right).$$
	Hence, $\sigma' \left(\left(\lambda \mathbf{1}-a\right)^{-1}x\right) = \sigma' \left(\left(\lambda \mathbf{1}-b\right)^{-1}x\right)$ for all rank one elements $x \in A$. Thus, $\mathrm{Tr}\,\left(\left(\lambda \mathbf{1}-a\right)^{-1}x\right) = \mathrm{Tr}\,\left(\left(\lambda \mathbf{1}-b\right)^{-1}x\right)$ for all rank one elements $x \in A$. Moreover, since
	$$\left(\lambda \mathbf{1}-a\right)^{-1} - \left(\lambda \mathbf{1}-b\right)^{-1} = \left(\lambda \mathbf{1}-a\right)^{-1}\left(a-b\right)\left(\lambda \mathbf{1}-b\right)^{-1} \in \mathrm{Soc}\:A,$$
	it follows as before from \cite[Corollary 3.6]{aupetitmoutontrace} that $\left(\lambda \mathbf{1}-a\right)^{-1} - \left(\lambda \mathbf{1}-b\right)^{-1} = 0$. Hence, $a = b$, which establishes the result.
\end{proof}

Let $J$ be a two-sided ideal of $A$. Denote by $l\left(J\right)$ the \textit{left-annihilator} of $J$, that is,
$$l\left(J\right) := \left\{x \in A: xJ=\left\{0\right\}  \right\}.$$
Similarly, we define the \textit{right-annihilator} of $J$ by
$$r\left(J\right) := \left\{x \in A: Jx=\left\{0\right\}  \right\}.$$

\begin{theorem}\label{3.3}
	Suppose that $\mathrm{Soc}\:A \neq \left\{0 \right\}$. Then $l\left(\mathrm{Soc}\:A\right) = \left\{0 \right\}$ if and only if the following are equivalent for any $a, b \in A$:
	\begin{itemize}
		\item[\textnormal{(i)}]
		$a = b$.
		\item[\textnormal{(ii)}]
		$\sigma (ax) = \sigma (bx)$ for all rank one elements $x \in A$.
		\item[\textnormal{(iii)}]
		$\sigma (a+x) = \sigma (b+x)$ for all rank one elements $x \in A$.
	\end{itemize}
\end{theorem}

\begin{proof}
	Suppose first that $l\left(\mathrm{Soc}\:A\right) = \left\{0 \right\}$ and let $a, b \in A$. Certainly, (i) $\Rightarrow$ (ii) and (i) $\Rightarrow$ (iii). Using the argument in the proof of Theorem \ref{3.2} we see that (ii) implies $\mathrm{Tr}\,\left(\left(a-b\right)y\right) = 0$ for all $y \in \mathrm{Soc}\:A$. Hence, by \cite[Corollary 3.6]{aupetitmoutontrace} it follows that $\left(a-b\right)\mathrm{Soc}\:A = \left\{0 \right\}$. Hence, $a-b \in l\left(\mathrm{Soc}\:A\right) = \left\{0 \right\}$, so (ii) $\Rightarrow$ (i). Similarly, the argument in the proof of Theorem \ref{3.2} can also be used to show that (iii) implies $\mathrm{Tr}\,\left(\left(\left(\lambda \mathbf{1}-a\right)^{-1}-\left(\lambda \mathbf{1}-b\right)^{-1}\right)y\right) = 0$ for all $y \in \mathrm{Soc}\:A$, where $\lambda \notin \sigma (a) \cup \sigma (b)$ is fixed. Hence, by \cite[Corollary 3.6]{aupetitmoutontrace} 
	$$\left(\lambda \mathbf{1}-a\right)^{-1}-\left(\lambda \mathbf{1}-b\right)^{-1} \in l\left(\mathrm{Soc}\:A\right) = \left\{0 \right\}.$$
	Thus, (iii) $\Rightarrow$ (i). This proves the forward implication. For the converse, we argue contrapositively. Suppose that $l\left(\mathrm{Soc}\:A\right) \neq \left\{0 \right\}$. Let $0 \neq a \in l\left(\mathrm{Soc}\:A\right)$ be fixed. Moreover, pick $y \in \mathrm{Soc}\:A$. Since $a \neq 0$, $a+y \neq y$. However, since $a\mathrm{Soc}\:A = \left\{0 \right\}$, it follows that $\sigma \left(\left(a+y\right)x\right) = \sigma (yx)$ for all rank one elements $x \in A$. Hence, (ii) $\not\Rightarrow$ (i). This completes the proof.
\end{proof}

In \cite[Theorem 2.5]{uniqueness} it was shown that properties (i) to (iii) are equivalent for any two bounded linear operators on a Banach space $X$. Consequently, Theorem \ref{3.3} implies the well-known fact that $l\left(\mathrm{Soc}\:B(X)\right) = \left\{0 \right\}$.

\begin{lemma}\label{3.4}
	Suppose that $\mathrm{Soc}\:A$ is a minimal two-sided ideal. Let $a, b \in A$ and suppose that $b = pt$, where $p = p^{2} \in \mathrm{Soc}\:A$ and $t \in G(A)$. If $\rho (ax) \leq \rho (bx)$ for all $x \in A$, then $a = \lambda b$ for some $\lambda \in \mathbb{C}$ with $\left|\lambda\right| \leq 1$. 
\end{lemma}

\begin{proof}
	If $p = 0$, then by semisimplicity $a = 0$ and we are done. So assume that $p \neq 0$. By hypothesis, $\rho \left(a'x\right) \leq \rho (px)$ for all $x \in A$, where $a' = at^{-1}$. It will suffice to show that $a'=\lambda p$ for some $\lambda \in \mathbb{C}$, since of course the assumption in conjunction with the Spectral Mapping Theorem automatically yields $\left|\lambda\right| \leq 1$. Replacing $x$ by $\left(\mathbf{1}-p\right)x$, we get $\rho \left(a'\left(\mathbf{1}-p\right)x\right) = 0$ for all $x \in A$. Thus, $a'\left(\mathbf{1}-p\right) \in \mathrm{Rad}\:A = \left\{0 \right\}$, and so, $a' = a'p$. Moreover, if we replace $x$ by $x\left(\mathbf{1}-p\right)$, then by Jacobson's Lemma we have $\rho \left(\left(\mathbf{1}-p\right)a'x\right) = 0$ for all $x \in A$. Hence, as before, the semisimplicity of $A$ yields $a'=pa'$. Consequently, $a' =  pa'p$. Now, by \cite[Lemma 2.5]{spectrumpreservers} $pAp$ is a closed semisimple subalgebra of $A$ with identity $p$. Moreover, $\sigma_{pAp}'(pxp) = \sigma_{A}'(pxp)$ for all $x \in A$. Hence, by hypothesis, we have
	$$\rho_{pAp} \left(\left(pa'p\right)\left(pxp\right)\right) \leq \rho_{pAp} \left(pxp\right) \;\,\mathrm{for\;all}\;\,x \in A.$$
	Hence, by the result in \cite{scalars} it follows that $a' \in Z(pAp)$. However, since $\mathrm{Soc}\:A$ is a minimal two-sided ideal, by \cite[Theorem 3.8, Theorem 3.9]{tracesocleident} we may infer that $pAp \cong M_{n} \left(\mathbb{C}\right)$. Consequently, $Z(pAp) = \mathbb{C}p$. So, $a' = \lambda p$ for some $\lambda \in \mathbb{C}$. The result now follows. 
\end{proof}

Let $p$ be a projection of $A$ with $\mathrm{rank}\,(p) \leq 1$. By $J_{p}$ we denote the two-sided ideal generated by $p$, that is, we let
$$J_{p} := \left\{ \sum_{j=1}^{n} x_{j}py_{j} : x_{j}, y_{j} \in A, n \geq 1\;\,\mathrm{an\;integer} \right\}.$$ 
By \cite[Lemma 2.2]{commutatorsinsocle} these $J_{p}$ are minimal two-sided ideals. Moreover, by \cite[Lemma 3.5]{tracesocleident}
there exists a collection of pairwise orthogonal two-sided ideals $\left\{J_{p} : p \in \mathcal{P} \right\}$ such that every element of  $\mathrm{Soc}\:A$ can be written as a finite sum of members of the $J_{p}$. In particular, this implies that $\mathrm{Soc}\:A$ is a minimal two-sided ideal whenever $A$ is prime. 

\begin{theorem}\label{3.5}
	$\mathrm{Soc}\:A$ is a minimal two-sided ideal if and only if the following are equivalent for any $a \in A$ and $b \in \mathrm{Soc}\:A$:
	\begin{itemize}
		\item[\textnormal{(i)}]
		$\rho (ax) \leq \rho (bx)$ for all $x \in A$
		\item[\textnormal{(ii)}]
		$a = \lambda b$ for some $\lambda \in \mathbb{C}$ with $\left|\lambda\right| \leq 1$.
	\end{itemize}
\end{theorem}

\begin{proof}
	Suppose first that $\mathrm{Soc}\:A$ is a minimal two-sided ideal and let $a \in A$ and $b \in \mathrm{Soc}\:A$. Obviously, (ii) $\Rightarrow$ (i), so assume that condition (i) holds. If $b = 0$, then by hypothesis and the semisimplicity of $A$ we have $a = 0$. So assume $b \neq 0$. By the Diagonalization Theorem and the density of $E(b)$ we can find mutually orthogonal rank one projections $p_{1}, \ldots, p_{n}$, $\alpha_{1}, \ldots, \alpha_{n} \in \mathbb{C}-\left\{0 \right\}$ and a $u \in G(A)$ such that $b = \alpha_{1}p_{1}u + \cdots +\alpha_{n}p_{n}u$. Observe firstly that if we set $p := p_{1} + \cdots +p_{n}$, then $p^{2}=p$ and $pb = b$. Consequently, by hypothesis and Jacobson's Lemma it follows that $\rho \left(\left(\mathbf{1}-p\right)ax\right) = 0$ for all $x \in A$. Hence, $\left(\mathbf{1}-p\right)a \in \mathrm{Rad}\:A = \left\{0 \right\}$, and so, $a = pa$. By orthogonality it follows that $\left(\alpha_{1}^{-1}p_{1} + \cdots + \alpha_{n}^{-1}p_{n}\right)b = pu$. Thus, by hypothesis and Jacobson's Lemma it follows that
	$$\rho \left(\left(\alpha_{1}^{-1}p_{1} + \cdots + \alpha_{n}^{-1}p_{n}\right)ax \right) \leq \rho \left(pux\right) \;\,\mathrm{for\;all}\;\,x \in A.$$
	Thus, by Lemma \ref{3.4} we may infer that
	$$ \left(\alpha_{1}^{-1}p_{1} + \cdots + \alpha_{n}^{-1}p_{n}\right)a = \lambda pu \;\,\mathrm{for\;some}\;\,\lambda \in \mathbb{C}.$$
	Hence,
	\begin{eqnarray*}
		a \;\,=\;\, pa & = & \left(\alpha_{1}p_{1} + \cdots +\alpha_{n}p_{n}\right)\left(\alpha_{1}^{-1}p_{1} + \cdots + \alpha_{n}^{-1}p_{n}\right)a \\
		& = & \left(\alpha_{1}p_{1} + \cdots +\alpha_{n}p_{n}\right) \left(\lambda pu\right) \\
		& = & \lambda \left(\alpha_{1}p_{1}u + \cdots +\alpha_{n}p_{n}u\right) \;\,=\;\, \lambda b.
	\end{eqnarray*}
	This proves the forward implication. For the reverse implication we argue contrapositively. Suppose that $\mathrm{Soc}\:A$ is not a minimal two-sided ideal. Then by \cite[Lemma 3.5]{tracesocleident} we may infer the existence of two rank one projections, say $p$ and $q$, such that $J_{p}J_{q} = J_{q}J_{p} = \left\{0 \right\}$. In particular, $p \neq \lambda \left(p+q\right)$ for all $\lambda \in \mathbb{C}$. Let $x \in A$ be arbitrary. Then $(px)(qx) = (qx)(px) = 0$. Hence, by \cite[Chapter 3, Exercise 9]{aupetit1991primer} it follows that $\sigma' ((p+q)x) = \sigma'(px)\cup \sigma'(qx)$. So, $\rho (px) \leq \rho ((p+q)x)$. Since $x \in A$ was arbitrary, this shows that (i) $\not\Rightarrow$ (ii), which completes the proof. 
\end{proof}

\begin{lemma}\label{3.6}
	Suppose that for any $a, b \in A$ we have that the following are equivalent:
	\begin{itemize}
		\item[\textnormal{(i)}]
		$\rho (ax) \leq \rho (bx)$ for all $x \in A$
		\item[\textnormal{(ii)}]
		$a = \lambda b$ for some $\lambda \in \mathbb{C}$ with $\left|\lambda\right| \leq 1$.
	\end{itemize}
	Then $A$ is a prime algebra.
\end{lemma}

\begin{proof}
	We shall argue contrapositively. If $A$ is not prime, then we can find two nonzero two-sided ideals $I$ and $J$ such that $IJ = \left\{0 \right\}$. Let $0 \neq a \in I$. If $a \in J$, then $aAa = \left\{0 \right\}$. But then, by semisimplicity, it follows that $a = 0$ which is absurd. Hence, $a \notin J$. Pick $0 \neq b \in J$. In particular then, $a \neq \lambda b$ for all $\lambda \in \mathbb{C}$. We firstly claim that $I \subseteq r(J)$. Let $x \in l(J)$ and let $y \in J$ be arbitrary. By Jacobson's Lemma and the fact that $J$ is a two-sided ideal, it follows that $\rho (yxw) = 0$ for all $w \in A$. Hence, $yx \in \mathrm{Rad}\:A = \left\{0 \right\}$. Since $y \in J$ was arbitrary, it follows that $I \subseteq r(J)$ as claimed. Since $a \neq \lambda b$ for all $\lambda \in \mathbb{C}$ and $b \neq 0$, we may infer that $a \neq \lambda\left(b + a\right)$ for all $\lambda \in \mathbb{C}$. Let $x \in A$ be arbitrary. Then $ax \in l(J) \cap r(J)$. Consequently, $(ax)(bx) = (bx)(ax) = 0$. Thus, by \cite[Chapter 3, Exercise 9]{aupetit1991primer} it follows that $\sigma'((a+b)x) = \sigma'(ax) \cup \sigma'(bx)$. Hence, $\rho (ax) \leq \rho ((a+b)x)$. Since $x \in A$ was arbitrary, this gives the result.
\end{proof}

\begin{theorem}\label{3.7}
	Let $A$ be a $C^{\ast}$-algebra. Then $A$ is prime if and only if for any $a, b \in A$ we have that the following are equivalent:
	\begin{itemize}
		\item[\textnormal{(i)}]
		$\rho (ax) \leq \rho (bx)$ for all $x \in A$
		\item[\textnormal{(ii)}]
		$a = \lambda b$ for some $\lambda \in \mathbb{C}$ with $\left|\lambda\right| \leq 1$.
	\end{itemize}
\end{theorem}

\begin{proof}
	This is immediate from \cite[Theorem 3.7]{detelements} and Lemma \ref{3.6}.
\end{proof}

\begin{lemma}\label{3.8}
	Suppose that $\mathrm{Soc}\:A$ is a minimal two-sided ideal and that $l\left(\mathrm{Soc}\:A\right) = \left\{0 \right\}$. Then for any $a, b \in A$ we have that the following are equivalent:
	\begin{itemize}
		\item[\textnormal{(i)}]
		$\rho (ax) \leq \rho (bx)$ for all $x \in A$
		\item[\textnormal{(ii)}]
		$a = \lambda b$ for some $\lambda \in \mathbb{C}$ with $\left|\lambda\right| \leq 1$.
	\end{itemize}
\end{lemma}

\begin{proof}
	Let $a, b \in A$. If $a = 0$, then we are done. So assume $a \neq 0$. It suffices to show that (i) $\Rightarrow$ (ii). Let $y \in \mathrm{Soc}\:A$ be arbitrary but fixed. By hypothesis, $\rho (ayx) \leq \rho (byx)$ for all $x \in A$. Hence, by Theorem \ref{3.5} there exists a $\lambda_{y} \in \mathbb{C}$ such that $ay = \lambda_{y}by$. Let $f_{b}: \mathrm{Soc}\:A \rightarrow \mathbb{C}$ and $f_{a}: \mathrm{Soc}\:A \rightarrow \mathbb{C}$ be defined as follows: $f_{b}(y) = \mathrm{Tr}\,(by)$ and $f_{a}(y) = \mathrm{Tr}\,(ay)$ for $y \in \mathrm{Soc}\:A$. Then $f_{b}$ and $f_{a}$ are nonzero linear functionals on the linear space $\mathrm{Soc}\:A$. Moreover, by our first observation it follows that $\mathrm{Ker}\:f_{b} \subseteq \mathrm{Ker}\:f_{a}$. Hence, from linear algebra (see \cite[p. 10]{diestel1984sequences}), it follows that $f_{a} = \lambda f_{b}$ for some $\lambda \in \mathbb{C}$. Thus, by the linearity of the trace, $\mathrm{Tr}\,\left(\left(a-\lambda b\right)y\right) = 0$ for all $y \in \mathrm{Soc}\:A$. Hence, by \cite[Corollary 3.6]{aupetitmoutontrace} it follows that $a - \lambda b \in l\left(\mathrm{Soc}\:A\right) = \left\{0 \right\}$ which gives the result. 
\end{proof}

\begin{theorem}\label{3.9}
	Suppose that $\mathrm{Soc}\:A \neq \left\{0 \right\}$. Then $A$ is prime if and only if for any $a, b \in A$ we have that the following are equivalent:
	\begin{itemize}
		\item[\textnormal{(i)}]
		$\rho (ax) \leq \rho (bx)$ for all $x \in A$
		\item[\textnormal{(ii)}]
		$a = \lambda b$ for some $\lambda \in \mathbb{C}$ with $\left|\lambda\right| \leq 1$.
	\end{itemize}
\end{theorem}

\begin{proof}
	The reverse implication follows immediately from Lemma \ref{3.6}. So assume that $A$ is prime. Since $\mathrm{Soc}\:A \neq \left\{0 \right\}$, we may infer that $l\left(\mathrm{Soc}\:A\right) = \left\{0 \right\}$. Moreover, since $A$ is prime, it readily follows from the remark preceding Theorem \ref{3.5} that $\mathrm{Soc}\:A$ is a minimal two-sided ideal. The forward implication therefore follows from Lemma \ref{3.8}.
\end{proof}

\begin{corollary}\label{3.10}
	Suppose that $\mathrm{Soc}\:A \neq \left\{0 \right\}$. Then $A$ is prime if and only if $\mathrm{Soc}\:A$ is a minimal two-sided ideal and $l\left(\mathrm{Soc}\:A\right) = \left\{0 \right\}$.
\end{corollary}

\begin{proof}
	This is a direct consequence of Lemma \ref{3.8} and Theorem \ref{3.9}.
\end{proof}

Let $0 \neq a \in A$ and $0 \neq b \in \mathrm{Soc}\:A$. It turns out that the condition
$$\sigma' (ax) \subseteq \sigma' (bx)\;\,\mathrm{for\;all}\;\,x\in A \Rightarrow a = b$$
can also be used to characterize socles which are minimal two-sided ideals. Firstly, however, we will prove some related results.\\

The next result was obtained by G. Braatvedt and R. Brits in \cite{uniqueness}. We state it together with a short new proof based on the spectral trace:

\begin{theorem}\cite[Corollary 2.3]{uniqueness}\label{3.a}
	Let $N$ be an arbitrary nonempty open subset of $A$ and let $a, b \in A$. If $\sigma (ax)$ and $\sigma (bx)$ are finite and equal for all $x \in N$, then $a = b$.
\end{theorem}

\begin{proof}
	Let $y \in A$. A standard argument using Baire's Category Theorem and the Scarcity Theorem can be used to show that if $\sigma (yx)$ is finite for all $x$ in some nonempty open set $N$ of $A$, then $y$ has finite rank. We may therefore infer that both $a$ and $b$ have finite rank. Furthermore, since $E(a)$ and $E(b)$ are both open dense subsets of $A$, it readily follows that $E(a) \cap E(b)$ is a dense subset of $A$. Consequently, we can find an $x_{0} \in N$ such that $ax_{0}$ and $bx_{0}$ are both maximal finite-rank elements. Let $y \in A$ be arbitrary but fixed, and define analytic functions from $\mathbb{C}$ into $\mathrm{Soc}\:A$ as follows:
	$$f(\lambda) = a\left[\left(1-\lambda\right)x_{0}+\lambda y\right] \;\,\mathrm{and}\;\,g(\lambda) = b\left[\left(1-\lambda\right)x_{0}+\lambda y\right]\;\,\left(\lambda \in \mathbb{C}\right).$$
	Since $\left(E(a) \cap E(b)\right) \cap N$ is a nonempty open set and $x_{0}$ belongs to this set, there exists a real number $\epsilon > 0$ such that for all $\lambda \in B\left(0, \epsilon\right)$ we have that $f\left(\lambda\right)$ and $g\left(\lambda\right)$ are maximal finite-rank elements and $\sigma \left(f\left(\lambda\right)\right) = \sigma \left(g\left(\lambda\right)\right)$. By the Diagonalization Theorem the functions
	$$\lambda \mapsto \mathrm{Tr}\,\left(f\left(\lambda\right)\right) \;\,\mathrm{and}\;\,\lambda \mapsto \mathrm{Tr}\,\left(g\left(\lambda\right)\right)$$
	agree on $B\left(0, \epsilon\right)$. Thus, since these functions are entire by \cite[Theorem 3.1]{aupetitmoutontrace}, it must be the case that they agree on all of $\mathbb{C}$. With the particular value $\lambda = 1$ we get $\mathrm{Tr}\,(ay) = \mathrm{Tr}\,(by)$. Since $y \in A$ was arbitrary we conclude by \cite[Corollary 3.6]{aupetitmoutontrace} that $a = b$. 
\end{proof}

From Theorem \ref{3.a} it is clear that if $\sigma (ax)$ and $\sigma (bx)$ are finite and equal for all $x$ in some nonempty open set $N$, then $\sigma (ax)=\sigma (bx)$ for all $x \in A$. In fact, we have the following result:

\begin{lemma}\label{3.f}
	Let $N$ be an arbitrary nonempty open subset of $A$ and let $a, b \in A$. If $\sigma (bx)$ is finite and $\sigma' (ax) \subseteq \sigma' (bx)$ for all $x \in N$, then $\sigma' (ax) \subseteq \sigma' (bx)$ for all $x \in A$. 
\end{lemma}

\begin{proof}
	The hypotheses allows us to infer that both $a$ and $b$ are finite-rank elements. Recall that $E(a)$ and $E(b)$ are open and dense in $A$. Hence, $E(a) \cap E(b)$ is open and dense in $A$. Fix any $x_{0} \in \left(E(a) \cap E(b)\right) \cap N$ and let $x \in A$ be arbitrary. Define the following analytic functions from $\mathbb{C}$ into $\mathrm{Soc}\:A$:
	$$f\left(\lambda\right) = a\left[\left(1-\lambda\right)x_{0} + \lambda x\right] \;\,\mathrm{and}\;\, g\left(\lambda\right) = b\left[\left(1-\lambda\right)x_{0} + \lambda x\right] \;\,\left(\lambda \in \mathbb{C}\right).$$
	Let $\mathrm{rank}\,(a) = k$ and $\mathrm{rank}\,(b) = n$ and note that $k \leq n$ (since $\left(E(a) \cap E(b)\right) \cap N \neq \emptyset$). By the Scarcity Theorem there exist two closed and discrete subsets of $\mathbb{C}$, say $F_{a}$ and $F_{b}$, such that $\#\sigma'\left(f\left(\lambda\right)\right) = k$ for all $\lambda \in \mathbb{C} - F_{a}$ and $\#\sigma'\left(g\left(\lambda\right)\right) = n$ for all $\lambda \in \mathbb{C} - F_{b}$. Moreover, by the Scarcity Theorem, our choice of $x_{0}$, and the definitions of $f$ and $g$, there exists a real number $\epsilon > 0$ such that for all $\lambda \in B\left(0, \epsilon\right)$, 
	$$\sigma'\left(f\left(\lambda\right)\right) = \left\{\alpha_{1} \left(\lambda\right), \ldots,\alpha_{k} \left(\lambda\right)  \right\},\; \sigma'\left(g\left(\lambda\right)\right) = \left\{\gamma_{1} \left(\lambda\right), \ldots,\gamma_{n} \left(\lambda\right)  \right\},$$ $\sigma'\left(f\left(\lambda\right)\right) \subseteq \sigma'\left(g\left(\lambda\right)\right)$, and the $\alpha_{i}$'s and $\gamma_{i}$'s are all holomorphic on $B\left(0, \epsilon\right)$. Let $i \in \left\{1, \ldots, k \right\}$ be arbitrary but fixed. We claim that $\alpha_{i} = \gamma_{j}$ for some $j \in \left\{1, \ldots, n \right\}$: Fix any $\beta_{0}$ in $B\left(0, \epsilon\right)$ and let $\left(\lambda_{m}\right)$ be any sequence in $B\left(0, \epsilon\right) - \left\{\beta_{0} \right\}$ which converges to $\beta_{0}$. Since $\sigma'\left(f\left(\lambda\right)\right) \subseteq \sigma'\left(g\left(\lambda\right)\right)$ for each $\lambda \in B\left(0, \epsilon\right)$, it follows that $\alpha_{i} \left(\lambda_{m}\right) = \gamma_{j_{m}}\left(\lambda_{m}\right)$ for some $j_{m} \in \left\{1, \ldots, n \right\}$. However, by the Pigeon Hole Principle we may infer the existence of a subsequence, denoted by $\left(\lambda_{m}\right)$ for convenience, and a $j \in \left\{1, \ldots, n \right\}$ such that $\alpha_{i} \left(\lambda_{m}\right) = \gamma_{j}\left(\lambda_{m}\right)$. However, then the set $\left\{\lambda \in B\left(0, \epsilon\right): \alpha_{i} \left(\lambda\right) - \gamma_{j}\left(\lambda\right) = 0 \right\}$ contains a limit point. So, from elementary Complex Analysis we conclude that $\alpha_{i} = \gamma_{j}$. This proves our claim. Without loss of generality we may therefore assume that  $\sigma'\left(f\left(\lambda\right)\right) = \left\{\gamma_{1} \left(\lambda\right), \ldots,\gamma_{k} \left(\lambda\right)  \right\}$ for each $\lambda \in B\left(0, \epsilon\right)$. Pick any $\lambda_{0} \in \partial B\left(0, \epsilon\right) \cap \left[\mathbb{C} - \left(F_{a} \cup F_{b}\right)\right]$ (which exists since $F_{a}$ and $F_{b}$ are discrete), and let $z \in \mathbb{C} - \left(F_{a} \cup F_{b}\right)$ be arbitrary. We claim that $\sigma'\left(f\left(z\right)\right) \subseteq \sigma'\left(g\left(z\right)\right)$: Since $F_{a}$ and $F_{b}$ are discrete, we can find a path $\Gamma$ in $\mathbb{C} - \left(F_{a} \cup F_{b}\right)$ which connects $\lambda_{0}$ and $z$. Now, for each $\lambda \in \Gamma$, there exists a nonempty open disk $B_{\lambda} := B\left(\lambda, r_{\lambda}\right)$ such that for $\beta \in B_{\lambda}$, $$\sigma'\left(f\left(\beta\right)\right) = \left\{\alpha^{\left(\lambda\right)}_{1} \left(\beta\right), \ldots,\alpha^{\left(\lambda\right)}_{k} \left(\beta\right)  \right\}\,\;\mathrm{and}\;\, \sigma'\left(g\left(\beta\right)\right) = \left\{\gamma^{\left(\lambda\right)}_{1} \left(\beta\right), \ldots,\gamma^{\left(\lambda\right)}_{n} \left(\beta\right)  \right\},$$
	where the $\alpha^{\left(\lambda\right)}_{i}$'s and $\gamma^{\left(\lambda\right)}_{i}$'s are all holomorphic on $B_{\lambda}$. By compactness we can find $\lambda_{1}, \ldots, \lambda_{m} \in \Gamma$ such that $B_{\lambda_{i}} \cap B_{\lambda_{i+1}} \neq \emptyset$ for $i \in \left\{0, \ldots, m-1 \right\}$ and $B_{\lambda_{m}} \cap B_{z} \neq \emptyset$. Now, observe that 
	$$\sigma'\left(f\left(\beta\right)\right) = \left\{\gamma_{1} \left(\beta\right), \ldots,\gamma_{k} \left(\beta\right)  \right\}\,\; \mathrm{and}\;\, \sigma'\left(g\left(\beta\right)\right) = \left\{\gamma_{1} \left(\beta\right), \ldots,\gamma_{n} \left(\beta\right)  \right\}$$
	for each $\beta \in B\left(0, \epsilon\right) \cap B_{\lambda_{0}}$. Since $B\left(0, \epsilon\right) \cap B_{\lambda_{0}}$ is a nonempty open set, it follows in a similar way as before that $$\sigma'\left(f\left(\beta\right)\right) = \left\{\gamma^{\left(\lambda_{0}\right)}_{1} \left(\beta\right), \ldots,\gamma^{\left(\lambda_{0}\right)}_{k} \left(\beta\right)  \right\}$$ for each $\beta \in B_{\lambda_{0}}$. Hence, $\sigma'\left(f\left(\beta\right)\right) \subseteq \sigma'\left(g\left(\beta\right)\right)$ for each $\beta \in B_{\lambda_{0}}$. Repeating this argument with the chain of intersecting open disks we may conclude that $\sigma'\left(f\left(\beta\right)\right) \subseteq \sigma'\left(g\left(\beta\right)\right)$ for each $\beta \in B_{z}$. This proves our claim. Since 
	$$z \in \mathbb{C} - \left(F_{a} \cup F_{b}\right)$$
	was arbitrary, $\sigma'\left(f\left(z\right)\right) \subseteq \sigma'\left(g\left(z\right)\right)$ for all $z \in \mathbb{C} - \left(F_{a} \cup F_{b}\right)$. Thus, by a straightforward argument, using the upper semicontinuity of the spectrum and Newburgh's Theorem \cite[Theorem 3.4.4]{aupetit1991primer}, we may conclude that the spectral containment extends to all of $\mathbb{C}$. Hence, 
	$$\sigma'(ax) = \sigma'\left(f\left(1\right)\right) \subseteq \sigma'\left(g\left(1\right)\right) = \sigma'(bx).$$ 
	Since $x \in A$ was arbitrary, this establishes the result.   
\end{proof}

The Jacobson radical formula is really only a particular case of a more general type of spectral calculus: Suppose $\sigma (bx)$ is finite for all $x \in A$. If for each $x \in A$ we have that $\sigma' (ax)$ is a portion of $\sigma' (bx)$, then ``$a$ is a portion of $b$'' in the following sense:  

\begin{theorem}\label{3.b}
	Let $N$ be an arbitrary nonempty open subset of $A$ and let $a, b \in A$. If $\sigma (bx)$ is finite for each $x \in N$, and $\sigma' (ax) \subseteq \sigma' (bx)$ for each $x \in N$ then $a$ commutes with $b$ and, either $a = 0$, or there exist rank one elements $a_{1}, \ldots, a_{n}$, and $k \leq n$ such that
	$$a = a_{1} + \cdots +a_{k}\;\,\mathrm{and}\;\,b = a_{1} + \cdots +a_{n}.$$
	Moreover, $a$ is orthogonal to $b-a$. 
\end{theorem}

\begin{proof}
	As before, by the hypotheses above, it follows that both $a$ and $b$ have finite rank. Moreover, by Lemma \ref{3.f} it follows that the spectral containment assumption ``for all $x \in N$'' may be replaced by ``for all $x \in A$''. Now, if $\sigma (ax) = \left\{0 \right\}$ for all $x \in A$, then by the semisimplicity of $A$ we may infer that $a = 0$. We may therefore assume that $a \neq 0$ and conclude that $\mathrm{rank}\,(b) = n \geq 1$. Recall that $E(a) \cap E(b)$ is an open dense subset of $A$ since $E(a)$ and $E(b)$ are both open and dense. Further, since $\sigma' (ax) \subseteq \sigma' (bx)$ for each $x \in A$, it follows in particular that $\mathrm{rank}\,(a) \leq \mathrm{rank}\,(b)$. Since $G(A)$ is open and $E(a) \cap E(b)$ is dense, we can fix an $x \in \left(E(a) \cap E(b)\right) \cap G(A)$. By the Diagonalization Theorem and our hypothesis on the spectrums of $ax$ and $bx$, we can find $n$ mutually orthogonal rank one projections $p_{1}, \ldots, p_{n}$, $k$ mutually orthogonal rank one projections $q_{1}, \ldots, q_{k}$ (with $k \leq n$), and nonzero complex numbers $\alpha_{1}, \ldots, \alpha_{n}$ such that
	\begin{equation}
	bx = \alpha_{1}p_{1} + \cdots +\alpha_{n}p_{n} \;\,\mathrm{and}\;\,
	ax = \alpha_{1}q_{1} + \cdots +\alpha_{k}q_{k}.
	\label{eq3.a}
	\end{equation}
	Set $b' = bx$ and $a'= ax$. Let $p$ be any rank one projection such that $a'p \neq 0$. Then $a'p$ has rank one. Moreover, by the containment above and the fact that $E\left(a'p\right)$ is dense, it follows that $\sigma' \left(a'py\right) = \sigma' \left(b'py\right)$ for all $y$ in a dense subset of $A$. Thus, $\mathrm{Tr}\,\left(a'py\right) = \mathrm{Tr}\,\left(b'py\right)$ for all $y$ in a dense subset of $A$. However, by \cite[Lemma 2.3]{tracesocleident} the trace is continuous on the set of rank one elements. Hence, $\mathrm{Tr}\,\left(a'py\right) = \mathrm{Tr}\,\left(b'py\right)$ for all $y \in A$, and so, $a'p = b'p$ by \cite[Corollary 3.6]{aupetitmoutontrace}. A similar statement is valid for multiplication on the left. We shall use this to show that $q_{j} = p_{j}$ for each $j \in \left\{1, \ldots, k \right\}$. For $j \in \left\{1, \ldots, k \right\}$ we have (cf. the remark following (\ref{beq1}))
	\begin{equation}
	q_{j} = \frac{1}{2\pi i} \int_{\Gamma_{j}} \left(\lambda \mathbf{1} - a'\right)^{-1}\,d\lambda
	\label{eq3.b}
	\end{equation}
	and
	\begin{equation}
	p_{j} = \frac{1}{2\pi i} \int_{\Gamma_{j}} \left(\lambda \mathbf{1} - b'\right)^{-1}\,d\lambda,
	\label{eq3.c}
	\end{equation}
	where $\Gamma_{j}$ is a small circle surrounding $\alpha_{j}$ and separating it from $0$ and the remaining spectrum of $b'$. From (\ref{eq3.a}) it follows that $q_{j}a' = a'q_{j} \neq 0$ so, by the preceding paragraph, we have $q_{j}b' = q_{j}a'$ and $b'q_{j} = a'q_{j}$ and hence that
	\begin{equation}
	q_{j}p_{j} = \frac{1}{2\pi i} \int_{\Gamma_{j}} q_{j} \left(\lambda \mathbf{1} - b'\right)^{-1}\,d\lambda = \frac{1}{2\pi i} \int_{\Gamma_{j}}q_{j} \left(\lambda \mathbf{1} - a'\right)^{-1}\,d\lambda = q_{j}^{2} = q_{j},
	\label{eq3.d}
	\end{equation}
	and similarly $p_{j}q_{j} = q_{j}$. Now, if $p_{j}a' = 0$ then
	$$p_{j}q_{j} = \frac{1}{2\pi i} \int_{\Gamma_{j}}p_{j} \left(\lambda \mathbf{1} - a'\right)^{-1}\,d\lambda = \frac{1}{2\pi i} \int_{\Gamma_{j}}\frac{p_{j}}{\lambda}\,d\lambda = 0$$
	which contradicts the first calculation that $p_{j}q_{j} = q_{j} \neq 0$. Thus, $p_{j}a' \neq 0$ from which we have $p_{j}a' = p_{j}b'$. From a similar argument we have $a'p_{j} = b'p_{j}$. As in (\ref{eq3.d}), but now using (\ref{eq3.b}), we have $p_{j}q_{j} = p_{j} = q_{j}p_{j}$. Hence, $ax = \alpha_{1}p_{1} + \cdots +\alpha_{k}p_{k}$. Since $x$ is invertible we can solve for $a$ and $b$ in (\ref{eq3.a}) and our result follows with $a_{j} = \alpha_{j}p_{j}x^{-1}$. Now, since $E(a) \cap E(b)$ is dense and open in $A$ we can find a sequence $\left(x_{n}\right) \subseteq E(a) \cap E(b)$ such that $x_{n} \rightarrow \mathbf{1}$ as $n \rightarrow \infty$. But for each $x_{n}$ the first part of the proof shows that
	$$ax_{n} \left(bx_{n} - ax_{n}\right) = \left(bx_{n} - ax_{n}\right)ax_{n} = 0.$$
	So in the limit we obtain $a(b-a) = (b-a)a = 0$ and hence also $ab=ba$.
\end{proof}

It is immediate from the above result that if we add to the assumptions the requirement that $\mathrm{rank}\,(a) = \mathrm{rank}\,(b)$, then $a = b$. With the hypothesis of Theorem \ref{3.b} $a$ inherits analytic properties from $b$:

\begin{corollary}\label{3.c}
	Suppose $a$ and $b$ satisfy the hypothesis of Theorem \textnormal{\ref{3.b}} and that $a \neq 0$. If $f(\lambda)$ is holomorphic on a domain $D$ containing $\sigma (b)$ and $f(b) = 0$, then also $f(a)=0$ and $f(b-a)=0$. In particular, if $b$ is a projection then so is $a$.
\end{corollary}

\begin{proof}
	If $b$ is invertible, then by Lemma \ref{3.f} and \cite[Theorem 2.1]{uniqueness} we have $a = \alpha b$ for some $\alpha \in \mathbb{C}$. So $\mathrm{rank}\,(a) = \mathrm{rank}\,(b)$, and the comment following Theorem \ref{3.b} readily yields $a = b$. We may therefore assume that $b \notin G(A)$ and that $\mathrm{rank}\,(a) < \mathrm{rank}\,(b) = n \neq 0$. Moreover, we may also assume that $f$ is not identically $0$. By hypothesis and the Spectral Mapping Theorem, $f\left(\sigma (b)\right) = \sigma \left(f(b)\right) = \left\{0 \right\}$. Hence, $f$ has zeroes at the spectral points of $b$. By \cite[Corollary 4.3.9]{conway1978functions} there exists a polynomial $h\left(\lambda\right)$ without a constant term and a holomorphic function $g\left(\lambda\right)$ on $D$ such that $f(\lambda) = h\left(\lambda\right)g\left(\lambda\right)$ and $g(\alpha) \neq 0$ for all $\alpha \in \sigma (b)$. In particular then, $g(b)$ is invertible by the Spectral Mapping Theorem. Hence, since $0 = f(b) = h(b)g(b)$ (by the Holomorphic Functional Calculus), we have $h(b) = 0$. Since $a$ and $b-a$ are orthogonal, it follows that $h(a) = -h(b-a)$. For the sake of a contradiction suppose that $h(a)$ and $h(b-a)$ are not zero. Then since $a = a_{1} + \cdots +a_{k+1}$ and $b-a = a_{k+1} + \cdots +a_{n}$ by Theorem \ref{3.b}, it follows that there is a largest integer $k+1 \leq m \leq n$ and $x_{1}, \ldots, x_{m} \in A$ such that
	$$0 \neq a_{m}x_{m} = a_{1}x_{1} + \cdots +a_{m-1}x_{m-1}.$$
	Since $a_{m}$ has rank one the minimal right ideal $a_{m}A = a_{m}x_{m}A$ which shows that $a_{m} \in a_{1}A + \cdots + a_{m-1}A$. However, by the subadditivity of the rank we then obtain that $\mathrm{rank}\,(b) < n$ which is absurd. Thus, $h(a) = h(b-a) = 0$ and the result follows from the Holomorphic Functional Calculus. The last part of the statement is obvious.
\end{proof}

The next result is similar in spirit to Theorem \ref{3.b}:

\begin{theorem}\label{3.d}
	Let $p$ be a projection of $A$, let $q \in A$, and suppose there exist a neighbourhood $N_{p}$ of $p$ and a neighbourhood $N_{\mathbf{1}-p}$ of $\mathbf{1}-p$ such that
	$$\#\sigma (qx) = \#\sigma (px) < \infty \,\;\mathrm{for\;all}\;\,x \in N_{p} \cup N_{\mathbf{1}-p}.$$
	Then $q$ is a scalar multiple of $p$ or $q$ is a scalar multiple of the identity.
\end{theorem}

\begin{proof}
	If $p = \mathbf{1}$, then by \cite[Theorem 2.1]{uniqueness} $q$ is a scalar multiple of the identity. If $p = 0$ and $q \notin G(A)$, then $q \in \mathrm{Rad}\:A = \left\{0 \right\}$. If $p = 0$ and $q \in G(A)$, then for all $x$ in a neighbourhood $N_{0}$ of $0$ we have that $\#\sigma (qx) = 1$ which by the Scarcity Theorem implies that every element of $A$ has one point spectrum. Thus, since $A$ is semisimple, $A\cong \mathbb{C}$ and hence $q$ is a scalar multiple of the identity. So assume that $p$ is neither $0$ nor $\mathbf{1}$, and that $q$ is not a scalar multiple of the identity. The hypothesis implies that, for all $x$ in some neighbourhood of $\mathbf{1}$, say $N_{1}$, we have $\#\sigma_{A} (qpxp) = \#\sigma_{A} (pxp)$. Moreover, since $yqpxp \in G(A)$ or $pxp \in G(A)$ implies $p = \mathbf{1}$ which contradicts our hypothesis on $p$, it follows that $0$ belongs to both $\sigma_{A} (yqpxp)$ and $\sigma_{A} (pxp)$ for all $x, y \in A$. Hence, by Jacobson's Lemma we may infer that $\#\sigma_{A} \left((pqp)(pxp)\right) = \#\sigma_{A} (pxp)$ for all $x \in N_{1}$. So it follows that
	$$  \#\sigma'_{pAp} \left((pqp)(pxp)\right) = \#\sigma'_{pAp} \left(p(pxp)\right) \;\,\mathrm{when}\;\,x \in N_{1}.$$
	Applying the Open Mapping Theorem to the continuous linear operator $x \mapsto pxp$ from $A$ onto $pAp$ we have that 
	$$  \#\sigma'_{pAp} \left((pqp)(pxp)\right) = \#\sigma'_{pAp} \left(p(pxp)\right)$$
	for all $pxp$ in some neighbourhood of $p$ in $pAp$. Hence, by the density of $E_{pAp}(pqp)$ and $E_{pAp}(p)$ in $pAp$ we may conclude that $\mathrm{rank}_{pAp} (pqp) = \mathrm{rank}_{pAp} (p)$. Whence, since $pAp$ is finite-dimensional, it follows that $pqp \in G(pAp)$. Thus,
	$$\#\sigma_{pAp} \left((pqp)(pxp)\right) = \#\sigma_{pAp} \left(p(pxp)\right)$$
	for all $pxp$ in some neighbourhood of $p$ in $pAp$. Hence, by \cite[Theorem 2.1]{uniqueness} it follows that $pqp = \alpha p$ for some $\alpha \in \mathbb{C}$. On the other hand, using the hypothesis with the neighbourhood $N_{\mathbf{1}-p}$ and the fact that $q$ is not a scalar multiple of the identity, it follows, for all $x$ in some neighbourhood of $\mathbf{1}$, that 
	$$\sigma \left(q\left(\mathbf{1}-p\right)x\right) = \sigma \left(\left(\mathbf{1}-p\right)qx\right) = \left\{0 \right\}.$$ 
	Hence, by the Scarcity Theorem and the semisimplicity of $A$ we get $q = pq = qp$. Therefore, $q = \alpha p$, which completes the proof. 
\end{proof}

\begin{theorem}\label{3.12}
	$\mathrm{Soc}\:A$ is a minimal two-sided ideal if and only if the following are equivalent for any $0 \neq a \in A$ and $0 \neq b \in \mathrm{Soc}\:A$:
	\begin{itemize}
		\item[\textnormal{(i)}]
		$\sigma' (ax) \subseteq \sigma' (bx)$ for all $x$ in some nonempty open set $N$.
		\item[\textnormal{(ii)}]
		$\sigma' (ax) \subseteq \sigma' (bx)$ for all $x \in A$.
		\item[\textnormal{(iii)}]
		$a = b$.
	\end{itemize}
\end{theorem}

\begin{proof}
	Suppose first that $\mathrm{Soc}\:A$ is a minimal two-sided ideal and let $0 \neq a \in A$ and $b \in \mathrm{Soc}\:A$. Obviously (iii) $\Rightarrow$ (i). Moreover, by Lemma \ref{3.f}, (i) $\Rightarrow$ (ii). So assume that condition (ii) holds. Since (ii) implies that $\rho (ax) \leq \rho (bx)$ for all $x \in A$, it readily follows from Theorem \ref{3.5} and hypothesis that $a = \lambda b$ for some $\lambda \in \mathbb{C}-\left\{0 \right\}$. Hence, $\mathrm{rank}\,(a) = \mathrm{rank}\,(b)$, and so, by Theorem \ref{3.b} and the remark following it, $a = b$. This proves the forward implication. For the other direction, we argue contrapositively. Suppose that $\mathrm{Soc}\:A$ is not a minimal two-sided ideal. Then by \cite[Lemma 3.5]{tracesocleident} we may infer the existence of two rank one projections $p$ and $q$ such that $J_{p}J_{q} = J_{q}J_{p} = \left\{0 \right\}$. However, as in the proof of Theorem \ref{3.5} this implies that $p \neq p+q$ and $\sigma' (px) \subseteq \sigma' \left(\left(p+q\right)x\right)$ for all $x \in A$. Hence, (ii) $\not\Rightarrow$ (iii), which establishes the result.
\end{proof}

Moreover, we obtain a similar characterization of prime Banach algebras as was done in Theorem \ref{3.9}:

\begin{theorem}\label{3.13}
	Suppose that $\mathrm{Soc}\:A \neq \left\{0 \right\}$. Then $A$ is prime if and only if for any $a, b \in A-\left\{0 \right\}$ we have that the following are equivalent:
	\begin{itemize}
		\item[\textnormal{(i)}]
		$\sigma' (ax) \subseteq \sigma' (bx)$ for all $x \in A$.
		\item[\textnormal{(ii)}]
		$a = b$.
	\end{itemize}
\end{theorem}

\begin{proof}
	If $A$ is not prime then we may proceed as in the proof of Lemma \ref{3.6} and expose two elements $a$ and $b$ such that $a \neq a+b$ and $\sigma' (ax) \subseteq \sigma' \left((a+b)x\right)$ for all $x \in A$. This proves the reverse implication. Conversely, if $A$ is prime, then since $\mathrm{Soc}\:A \neq \left\{0 \right\}$ it follows that $\mathrm{Soc}\:A$ is a minimal two-sided ideal and that $l\left(\mathrm{Soc}\:A\right) = \left\{0 \right\}$. Let $a, b \in A-\left\{0\right\}$ be arbitrary. Obviously (ii) $\Rightarrow$ (i). So assume that condition (i) holds and let $y \in \mathrm{Soc}\:A$ be arbitrary but fixed. Then by Theorem \ref{3.12} we may infer that $ay = by$. Since $y \in \mathrm{Soc}\:A$ was arbitrary, we conclude that $\mathrm{Tr}\,\left((a-b)y\right) = 0$ for all $y \in \mathrm{Soc}\:A$. Hence, by \cite[Corollary 3.6]{aupetitmoutontrace} it follows that $a-b \in l\left(\mathrm{Soc}\:A\right) = \left\{0 \right\}$. Therefore, (i) $\Rightarrow$ (ii), so the theorem is true.  
\end{proof}

To conclude we will show that if $\mathrm{Soc}\:A$ is a minimal two-sided ideal, then conditions (i) and (ii) in Theorem \ref{3.13} are equivalent whenever $b$ belongs to some \textit{inessential ideal}; that is, a two-sided ideal in which the spectrum of all elements contain at most $0$ as an accumulation point. Before that, however, we will need a little preparation:

\begin{lemma}\label{3.14}
	Let $s \in A$ and for each $x \in A$ suppose that $\sigma (sx)$ contains at most $0$ as an accumulation point for all $x \in A$. Then the Riesz projections of $s$ corresponding to nonzero spectral values have finite rank.
\end{lemma}

\begin{proof}
	Let $\sigma'(s)=\{\lambda_1,\lambda_2,\dots\}$ and set, for $i\in\mathbb N$, $p:=p(\lambda_i,s)$. Recall that $pAp$ is a semisimple Banach algebra with identity $p$. There exists an open neighborhood $V$ of $\mathbf 1$ in $A$ such that $pxp$ is invertible in $pAp$ for each $x\in V$. Now suppose $x\in V$ and $\#\sigma_A(px)=\infty$. Then, by Jacobson's Lemma, $\#\sigma_A(pxp)=\infty=\#\sigma_{A}'(pxp)$, and, since $p \in sA$, it follows from our hypothesis on $s$ that $\sigma_{A}'(pxp)$ is a sequence converging to $0$. But this means
	$\sigma_{pAp}(pxp)$ contains a sequence converging to zero, from which it follows (since the spectrum is closed) that $pxp$ cannot be invertible in $pAp$ giving a contradiction.
	So $\#\sigma_A(px)<\infty$ for all $x\in V$ and a standard application of the Scarcity Theorem then says $\#\sigma_A(px)<\infty$ for all $x\in A$. Thus $\mathrm{rank}\,(p)<\infty$.
\end{proof}

\begin{theorem}\label{3.15}
	Suppose that $\mathrm{Soc}\:A$ is a minimal two-sided ideal. Let $0 \neq a \in A$ and let $0 \neq b \in A$ such that $\sigma (bx)$ has at most $0$ as an accumulation point for all $x \in A$. Then the following are equivalent:
	\begin{itemize}
		\item[\textnormal{(i)}]
		$\sigma' (ax) \subseteq \sigma' (bx)$ for all $x \in A$.
		\item[\textnormal{(ii)}]
		$a = b$.
	\end{itemize}
\end{theorem}

\begin{proof}
	Let $0 \neq a \in A$ and $b \in A$. Surely, (ii) $\Rightarrow$ (i). So assume that condition (i) holds. We claim that $\sigma (ax) = \sigma (bx)$ for all $x \in A$: Let $x \in A$ be arbitrary. It will suffice to show that $\sigma' (ax) = \sigma' (bx)$ and $0 \in \sigma (ax) \Leftrightarrow 0 \in \sigma (bx)$. If $\sigma (bx) = \left\{0 \right\}$, then $\sigma' (ax) = \sigma' (bx) = \emptyset$. So assume that $\sigma (bx) \neq \left\{0 \right\}$ and let $\lambda \in \sigma' (bx)$. Since $\sigma' (bx)$ is either finite or a sequence converging to zero, we may consider the Riesz projection of $bx$ associated with $\lambda$, say $p := p\left(\lambda, bx\right)$. Now, by Lemma \ref{3.14} it follows that $p \in \mathrm{Soc}\:A$. Consequently, by hypothesis and Theorem \ref{3.12}, we may infer that $axp = bxp$. Moreover, since $bxp = pbx$, by condition (i), Jacobson's Lemma and Theorem \ref{3.12} it follows that $pax = pbx = axp$. Hence,
	$$\left(ax\left(\mathbf{1}-p\right)\right)(axp) = (axp)\left(ax\left(\mathbf{1}-p\right)\right) = 0.$$
	Thus, since $ax = ax\left(\mathbf{1}-p\right) + axp$, it follows from \cite[Chapter 3, Exercise 9]{aupetit1991primer} that
	$$\sigma' (ax) = \sigma'\left(ax\left(\mathbf{1}-p\right)\right) \cup \sigma' (axp) =  \sigma'\left(ax\left(\mathbf{1}-p\right)\right) \cup \sigma' (bxp).$$
	But by the Holomorphic Functional Calculus it follows that $\sigma' (bxp) = \left\{\lambda \right\}$. Hence, $\lambda \in \sigma' (ax)$. This shows that $\sigma' (ax) = \sigma' (bx)$. Suppose now that $0 \notin \sigma (bx)$. Then, by hypothesis on $b$ it must be the case that $\sigma (bx)$ is finite, say $\sigma (bx) = \left\{\alpha_{1}, \ldots, \alpha_{r} \right\}$. For each $i \in \left\{1, \ldots, r \right\}$, let $p_{i}$ denote the Riesz projection of $bx$ associated with $\alpha_{i}$. By condition (i) and Theorem \ref{3.12} it follows that $axp_{i} = bxp_{i}$ for all $i \in \left\{1, \ldots, r \right\}$. But by the Holomorphic Functional Calculus $p_{1} + \cdots +p_{r} = \mathbf{1}$. Hence,
	\begin{eqnarray*}
		ax & = & ax \left(p_{1} + \cdots +p_{r} \right) = axp_{1} + \cdots +axp_{r} \\
		& = & bxp_{1} + \cdots +bxp_{r} = bx \left(p_{1} + \cdots +p_{r} \right) \,\;=\,\; bx.
	\end{eqnarray*}
	So, $0 \notin \sigma (ax)$. Similarly, $0 \notin \sigma (ax)$ yields $bx = ax$ and consequently $0 \notin \sigma (bx)$. Hence, $0 \in \sigma (ax) \Leftrightarrow 0 \in \sigma (bx)$. This proves our claim. By Theorem \ref{1.1} we may therefore conclude that $a = b$, which completes the proof.    
\end{proof}

\bibliographystyle{amsplain}
\bibliography{Spectral}

\providecommand{\bysame}{\leavevmode\hbox to3em{\hrulefill}\thinspace}
\providecommand{\MR}{\relax\ifhmode\unskip\space\fi MR }
\providecommand{\MRhref}[2]{%
  \href{http://www.ams.org/mathscinet-getitem?mr=#1}{#2}
}
\providecommand{\href}[2]{#2}
\begin{thebibliography}{10}

\bibitem{aupetit1991primer}
B.~Aupetit, \emph{A {P}rimer {O}n {S}pectral {T}heory}, Universitext (1979),
  Springer-Verlag, 1991.

\bibitem{spectrumpreservers}
B.~Aupetit, \emph{Spectrum {P}reserving {L}inear {M}appings between {B}anach
  {A}lgebras or {J}ordan {B}anach {A}lgebras}, J. London Math. Soc. \textbf{62}
  (2000), 917--924.

\bibitem{aupetitmoutontrace}
B.~Aupetit and H.~du~T.~Mouton, \emph{Trace and {D}eterminant in {B}anach
  {A}lgebras}, Studia {M}ath. \textbf{121} (1996), 115--136.

\bibitem{bonsall1973complete}
F.F. Bonsall and J.~Duncan, \emph{Complete {N}ormed {A}lgebras}, Ergeb. Math.
  Grenzgeb., Springer-Verlag, 1973.

\bibitem{uniqueness}
G.~Braatvedt and R.~Brits, \emph{Uniqueness and {S}pectral {V}ariation in
  {B}anach {A}lgebras}, Quaestiones Math. \textbf{36} (2013), 155--165.

\bibitem{scalars}
G.~Braatvedt, R.~Brits, and H.~Raubenheimer, \emph{Spectral {C}haracterizations
  of {S}calars in a {B}anach {A}lgebra}, Bull. {L}ond. {M}ath. {S}oc.
  \textbf{41} (2009), 1095--1104.

\bibitem{detelements}
M.~Bre\v{s}ar and \v{S}. \v{S}penko, \emph{Determining {E}lements in {B}anach
  {A}lgebras through {S}pectral {P}roperties}, J. {M}ath. {A}nal. {A}ppl.
  \textbf{393} (2012), 144--150.

\bibitem{conway1978functions}
J.B. Conway, \emph{Functions of {O}ne {C}omplex {V}ariable {I}}, Springer,
  1978.

\bibitem{diestel1984sequences}
J.~Diestel, \emph{Sequences and {S}eries in {B}anach {S}paces}, Graduate texts
  in mathematics, Springer-Verlag, 1984.

\bibitem{kaplansky1970algebraic}
I.~Kaplansky, \emph{Algebraic and {A}nalytic {A}spects of {O}perator
  {A}lgebras}, Regional {C}onference {S}eries in {M}athematics 1, Amer. {M}ath.
  {S}oc., 1970.

\bibitem{commutatorsinsocle}
F.~Schulz and R.~Brits, \emph{Commutators, {C}ommutativity and {D}imension in
  the {S}ocle of a {B}anach {A}lgebra: {A} generalized {W}edderburn-{A}rtin and
  {S}hoda's {T}heorem}, Linear {A}lgebra {A}ppl. \textbf{484} (2015), 175--198.

\bibitem{tracesocleident}
F.~Schulz, R.~Brits, and G.~Braatvedt, \emph{Trace {C}haracterizations and
  {S}ocle {I}dentifications in {B}anach {A}lgebras}, Linear {A}lgebra {A}ppl.
  \textbf{{472}} (2015), 151--166.

\end{thebibliography}

\end{document}